\documentclass[11pt]{amsart}
\usepackage{amsfonts}
\usepackage{amsmath,amsthm,latexsym,amssymb,graphicx}
\usepackage{graphicx,color,enumerate}

\numberwithin{equation}{section}

\newtheorem{theorem}{Theorem}
\newtheorem{lemma}{Lemma}

\newtheorem{proposition}{Proposition}

\numberwithin{theorem}{section}
\numberwithin{corollary}{section}
\numberwithin{lemma}{section}
\numberwithin{definition}{section}
\numberwithin{proposition}{section}
\numberwithin{remark}{section}

\newcommand{\dint}{\ds\int}
\newcommand{\ds}{\displaystyle}

\newcommand{\R}{\mathbb R}

\newcommand{\medint}{-\kern  -,375cm\int}

\begin{document}
\title[ ]{Optimal lower bounds for eigenvalues \\
of linear and nonlinear Neumann problems}
\author[B. Brandolini, F. Chiacchio, C. Trombetti]{B. Brandolini$^*$ - F.
Chiacchio$^*$ - C. Trombetti$^*$}
\thanks{$^*$ Dipartimento di Matematica e Applicazioni ``R. Caccioppoli'',
Universit\`{a} degli Studi di Napoli ``Federico II'', Complesso Monte S.
Angelo, via Cintia - 80126 Napoli, Italy; email: brandolini@unina.it;
francesco.chiacchio@unina.it; cristina@unina.it}

\begin{abstract}
In this paper we prove a sharp lower bound for the first nontrivial Neumann eigenvalue $\mu_1(\Omega)$ for the $p$-Laplace operator in a Lipschitz, bounded domain $\Omega$ in $\R^n$. Our estimate does not require any convexity assumption on $\Omega$ and it involves the best isoperimetric constant relative to $\Omega$. 
\end{abstract}

\maketitle

\section{Introduction}

In this paper we provide sharp lower bounds for the first nontrivial
eigenvalue $\mu_1(\Omega)$ of the $p$-Laplacian operator with homogeneous
Neumann boundary conditions in a Lipschitz, bounded domain $\Omega$ of $\mathbb{R}^n$. 
Hence we deal with the following eigenvalue problem 
\begin{equation}  \label{problem}
\left\{ 
\begin{array}{ll}
-\Delta_pu=\mu |u|^{p-2}u & \mathrm{{in}\> \Omega} \\ 
&  \\ 
\dfrac{\partial u}{\partial \nu}=0 & \mathrm{{on}\> \partial \Omega,}
\end{array}
\right.
\end{equation}
where $\nu$ is the outward normal to $\partial \Omega$. It is well-known that $\mu_1(\Omega)$ can be characterized as follows
\begin{equation*}
\mu_1(\Omega)={\rm min} \left\{ \dfrac{\displaystyle\int_\Omega |Du|^p \, dx}{\displaystyle\int_\Omega |u|^p \, dx} :  u\in W^{1,p}(\Omega)\setminus \{0\}, \,\int_\Omega|u|^{p-2}u\, dx=0 \right\}.
\end{equation*}
Moreover $\mu_1(\Omega)^{-1/p}$ is the best constant in the following Poincar\' e inequality
\begin{equation*}
\inf_{t\in \R} \|u-t\|_{L^p(\Omega)}\le C_{\Omega,p} \|Du\|_{L^p(\Omega)}, \quad u\in W^{1,p}(\Omega).
\end{equation*}

In the celebrated paper \cite{PW} the authors prove that, when $p=2$ and $
\Omega$ is convex with diameter $d(\Omega)$ (see also \cite{V,ENT,FNT}) then 
\begin{equation}\label{PW}
\mu_1(\Omega) \ge \frac{\pi^2}{d(\Omega)^2}. 
\end{equation}
The above estimate is asymptotically sharp since  $\mu_1(\Omega){d(\Omega)^2}$ tends to $\pi^2$ for a
parallelepiped all but one of whose dimensions shrink to zero. On the other hand,
Payne-Weinberger estimate does not hold true in general for non convex sets.

Here we allow the set to be non convex and in place of the diameter our
estimate will involve $K_n(\Omega)$, the best isoperimetric constant   
relative to $\Omega$, that is 
\begin{equation}  \label{reliso}
K_n(\Omega)=\inf_{E \subset \Omega}\frac{P_\Omega(E)}{(\min\{|E|,|\Omega
\setminus E|\})^{1-1/n}},
\end{equation}
where $P_\Omega(E)$ is the perimeter of $E$ relative to $\Omega$ and $|\cdot|
$ stands for the $n$-dimensional Lebesgue measure. Obviously $K_n(\Omega)\le
K_n(\mathbb{R}^n)=n\omega_n^{1/n}$ (the classical isoperimatric constant),
where $\omega_n$ is the measure of the unitary ball in $\mathbb{R}^n$. On
the other hand $K_n(\Omega)>0$ since we are assuming that $\Omega$ has a
Lipschitz boundary (see \cite{M}). 

Let $\lambda_1(\Omega^\sharp)$ be the first Dirichlet eigenvalue of the ball 
$\Omega^\sharp$ having the same measure as $\Omega$. Our main result is the
following

\begin{theorem}
\label{main} Let $\Omega$ be a Lipschitz domain of $\mathbb{R}^n$. Then 
\begin{equation}  \label{estimate}
\mu_1(\Omega) \ge2^{p/n}\left(\dfrac{K_n(\Omega)}{K_n(\mathbb{R}^n)}
\right)^p\lambda_1(\Omega^\sharp).
\end{equation}
Furthermore \eqref{estimate} is sharp at least in the case $n=p=2$.
\end{theorem}

We remark that for $p = 2$ \eqref{estimate} improves previous results
contained in \cite{BCT}, while in the case $p > 2$ we  prove (see Section 3) that 
\eqref{estimate} is better than the ones already available in literature
(see \cite{A,AM}).

In order to prove \eqref{estimate} we consider an eigenfunction $u_1$
corresponding to $\mu_1(\Omega)$ such that $|\mathrm{supp}
\left(u_1^+\right)|\le \dfrac{|\Omega|}{2}$, where
$u_1^+(x) = \max \{u_1(x), 0\}$.
 Then we prove a comparison
result \`a la Chiti and in turn a Payne-Rainer type inequality for for $u_1^+$ (see \cite{Ch} and also \cite{AC,BCF, BCT,BT,DPG}).
Namely, we show that 
\begin{equation*}
||u_1^+||_{L^q(\Omega)} \le C ||u_1^+||_{L^r(\Omega)}, \qquad 0<r<q<+\infty 
\end{equation*}
where $C$ is positive constant whose value depends on $n,p,q,r,K_n(\Omega),
\mu_1(\Omega)$ and is explicitly given in Section 2. This technical result together with a limit as $q\to r\to 0$ 
will be the key ingredients in the proof of Theorem \ref{main}.
Finally, when $p=n=2$, we consider a sequence of rhombi $\Omega_m$ of side 1
and acute angle $\beta_m=\frac{2\pi}{m}$ ($m \ge 4$). In our previous paper 
\cite{BCT} we proved a reverse H\"older inequality for $u_1$ that becomes
asymptotically sharp on $\Omega_m$. Here we show that estimate 
\eqref{estimate} is asymptotically sharp along the same sequence of domains.

Finally, for the interested reader, other estimates for eigenvalues of Neumann problems can be found for instance in \cite{BPV, BNT, LS, BCM, CdB, BCT1, BCTH}.

\section{A Payne-Rainer type inequality}

In this section we prove a reverse H\"older inequality for an eigenfunction 
$u_1$ corresponding to $\mu_1(\Omega)$. To this aim we recall some notation
about rearrangements and we provide some auxiliary lemmata.

Let $\Omega $ be a bounded, open set in $\mathbb{R}^{n}$ and let $u$ be a
measurable real function defined in $\Omega $. The distribution function of 
$u$ is defined by 
\begin{equation*}
m(t)=|\{x\in \Omega :\>u(x)>t\}|, \qquad t \in \mathbb{R}, 
\end{equation*}
while the decreasing rearrangement of $u$ is the function 
\begin{equation*}
u^{\star }(s)=\sup \left\{ t\in \mathbb{R}:\>m(t)> s\right\} ,\quad s\in (
0,|\Omega |).
\end{equation*}
It is easy to see that $u^{\star }$ is a non increasing, right-continuous
function defined in $(0,|\Omega |)$, equidistributed with $u$, that means
that $u$ and $u^{\star }$ have corresponding superlevel sets with the same
measure. This feature implies that $u$ and $u^{\star }$ have the same $L^{p}$
norms 
\begin{equation*}
||u||_{L^{p}(\Omega )}=||u^{\star }||_{L^{p}(0,|\Omega |)},\quad \forall
p\geq 1,
\end{equation*}
and, clearly, 
\begin{equation*}
\int_{\Omega }u\,dx=\int_{0}^{|\Omega |}u^{\star }(t)dt.
\end{equation*}
\noindent For an exhaustive treatment on rearrangements see, for instance, 
\cite{Ka,T,H,Ke}.

The symmetrization procedure we will adopt will lead us to consider a
one-dimensional Sturm-Liouville problem of the type 
\begin{equation}\label{ode}
\left\{ 
\begin{array}{ll}
-\left(|\psi'(s)|^{\gamma-2}\psi'(s)\right)^{\prime}=\sigma |\psi(s)|^{\gamma-2}\psi(s)s^{-\beta}, \qquad s
\in (0,A) 
&  \\ \\
\psi(0)=\psi^{\prime }(A)=0, & 
\end{array}
\right.
\end{equation}
with $\gamma \geq 1$ and  $\beta > 0$.
We consider the functional space naturally associated to \eqref{ode}
$$
\mathcal{W}=\left\{ \phi \in W^{1,\gamma}(0,A):
\phi(0)=0  \right\},
$$
endowed with the norm $||\phi||_{\mathcal{W}}=\left(\int_{0}^{A}\left\vert \phi ^{\prime }(s)\right\vert ^{\gamma }ds\right)^{1/\gamma}$, and 
the weighted Lebesgue space
$$
L^{\gamma }( (0,A);s^{-\beta }) 
=
\left\{ \phi :\left[ 0,A\right]\rightarrow \mathbb{R}:
\left\Vert \phi \right\Vert_{L^{\gamma }((0,A);s^{-\beta }) }
=
\left(\int_{0}^{A}\left\vert \phi \right\vert^{\gamma }s^{-\beta }ds\right)^{1/\gamma}<\infty \right\} .
$$
A result contained in \cite{N} ensures that $\mathcal{W}$  is compactly embedded in 
$L^{\gamma }( (0,A);s^{-\beta }) $. Here, for the reader's convenience, we provide a simple proof based on the one dimensional Hardy inequality (see also \cite{C} for the linear case). 

\begin{lemma}
\label{lemma} Let $A>0,  \gamma>1$ and $0<\beta<\gamma$. Then $\mathcal{W}$ is compactly embedded in $L^\gamma((0,A);s^{-\beta})$.
\end{lemma}

\begin{proof}
Let $\phi \in \mathcal{W}$; by Hardy inequality it holds 
\begin{equation*}
\int_0^A\frac{|\phi(s)|^\gamma}{s^\beta}\,ds \le A^{\gamma-\beta}\int_0^A \frac{
|\phi(s)|^\gamma}{s^\gamma}\,ds\le A^{\gamma-\beta}\left(\frac{\gamma}{\gamma-1}
\right)^\gamma \int_0^A|\phi'(s)|^\gamma\,ds, 
\end{equation*}
that is $\mathcal{W}$ is continuously embedded in $L^\gamma((0,A);
s^{-\beta})$. Now consider a
sequence $\{\phi_m \}_{m\in \mathbb{N}} \subset \mathcal{W}$ such
that $||\phi_m||_{\mathcal{W}}\le 1.$ By classical results on Sobolev spaces there
exists $\phi \in \mathcal{W}$ such that, up to a subsequence, 
$$
\phi_m \to \phi \quad \mbox{a.e.},\quad\quad \phi_m \rightharpoonup \phi \quad \mbox{in}\>\> \mathcal{W}, \quad\quad \phi_m \to \phi \quad \mbox{in}\>\> L^{\gamma}(0,A).
$$
We claim that $\phi_m \to \phi$ in $L^\gamma((0,A); s^{-\beta})$. Fix 
$\epsilon >0$ and let $m$ be large enough to ensure 
\begin{equation*}
\int_0^A|\phi_m(s)-\phi(s)|^\gamma\,ds < \epsilon^{\gamma+\beta}. 
\end{equation*}
Then 
$$
\int_0^A\frac{|\phi_m(s)-\phi(s)|^\gamma}{s^\beta} \,ds
\le \epsilon^{\gamma-\beta}\int_0^\epsilon \frac{|\phi_m(s)-\phi(s)|^\gamma}{
s^\gamma}\,ds+\epsilon^{-\beta}\int_\epsilon^A|\phi_m(s)-\phi(s)|^\gamma\,ds 
\le C\epsilon^{\gamma-\beta},
$$
where $C$ is a positive constant whose value does not depend on $m$.
\end{proof}

\noindent From Lemma \ref{lemma} we immediately deduce that, when $\gamma >1$ and  $0<\beta<\gamma$, the first
eigenvalue $\sigma_1(0,A)$ of problem \eqref{ode} can be variationally characterized as follows 
\begin{equation*}
\sigma_1(0,A)=\min\left\{\frac{\int_0^A |\phi'(s)|^{\gamma}ds}{\int_0^A
|\phi(s)|^\gamma s^{-\beta}ds}:\> \phi \in
\mathcal{W}\setminus\{0\}\right\}.
\end{equation*}
Moreover it is simple (see for instance Theorem 2.3 in \cite{AFT}).

Now we can turn our attention on the original problem \eqref{problem}. From now on we will assume that $\Omega$ is a bounded, Lipschitz domain of $\R^n$.

\begin{lemma}
Let $u_1$ be an eigenfunction corresponding to $\mu_1(\Omega)$. Then the
following inequalities hold 
\begin{eqnarray}
\left(-u_1^\star(s)^{\prime }\right)^{p-1} &\le& \frac{\mu_1(\Omega)}{
K_n(\Omega)^p} \>s^{-p+p/n}\int_0^s |u_1^\star(t)|^{p-2}u_1^\star(t)\,dt,
\qquad s\le \frac{|\Omega|}{2} ,  \label{1star} \\
\left(-u_1^\star(s)^{\prime }\right)^{p-1} &\le& \frac{\mu_1(\Omega)}{
K_n(\Omega)^p} \> \left(|\Omega|-s\right)^{-p+p/n}\int_0^s
|u_1^\star(t)|^{p-2}u_1^\star(t)\,dt ,\qquad s>\frac{|\Omega|}{2}.
\label{2star}
\end{eqnarray}
\end{lemma}

\begin{proof}
Fix $t \in \mathbb{R}$ and let $h>0$; we choose 
\begin{equation*}
\phi_h(x)=\left\{
\begin{array}{ll}
h & \mbox{if}\>\> u_1(x)>t+h \\ 
u_1(x)-t & \mbox{if}\>\> t < u_1(x) \le t+h \\ 
0 & \mbox{if}\>\> u_1(x) \le t
\end{array}
\right. 
\end{equation*}
as test function in \eqref{problem} and we obtain 
\begin{equation*}
\frac{1}{h}\int_{t<u_1\le
t+h}|Du_1|^p\,dx=\mu_1(\Omega)\int_{u_1>t+h}|u_1|^{p-2}u_1\,dx+\frac{
\mu_1(\Omega)}{h}\int_{t<u_1\le t+h}|u_1|^{p-2}u_1(u_1-t)\,dx. 
\end{equation*}
Letting $h \to 0^+$ we get that for almost every $t\in \mathbb{R}$ 
\begin{eqnarray*}
\frac{1}{h}\int_{t<u_1\le t+h}|Du_1|^p\,dx & \longrightarrow& -\frac{d}{dt}
\int_{u_1>t}|Du_1|^p\,dx \\
\int_{u_1>t+h}|u_1|^{p-2}u_1\,dx & \longrightarrow& \int_{u_1>t}
|u_1|^{p-2}u_1\,dx \\
\left| \frac{1}{h}\int_{t<u_1\le t+h}|u_1|^{p-2}u_1(u_1-t)\right| \,dx&\le&
\int_{t<u_1\le t+h} |u_1|^{p-1}\,dx \longrightarrow 0.
\end{eqnarray*}
Thus 
\begin{equation*}
-\frac{d}{dt}\int_{u_1>t}|Du_1|^p\,dx =
\mu_1(\Omega)\int_{u_1>t}|u_1|^{p-2}u_1\,dx,\quad \mathrm{{for\> a.e.}\> t
\in \mathbb{R}. }
\end{equation*}
On the other hand, by co-area formula and H\"older inequality, it holds 
\begin{equation}  \label{u}
P_\Omega \left(\{u_1>t\}\right)=-\frac{d}{dt}\int_{u_1>t}|Du_1|\,dx \le
\left(-m^{\prime }(t)\right)^{1-1/p}\left(-\frac{d}{dt}\int_{u_1>t}|Du_1|^p
\,dx\right)^{1/p}.
\end{equation}
Let $t_0=\inf\left\{t \in \mathbb{R}:\> m(t) \le \frac{|\Omega|}{2}\right\}$
; by the very definition \eqref{reliso} of $K_n(\Omega)$ we get 
\begin{equation}  \label{p}
P_\Omega \left(\{u_1>t\}\right) \ge \left\{
\begin{array}{ll}
K_n(\Omega) m(t)^{1-1/n} & \mbox{if}\> t \ge t_0 \\ 
&  \\ 
K_n(\Omega) \left(|\Omega|-m(t)\right)^{1-1/n} & \mbox{if}\> t < t_0.
\end{array}
\right.
\end{equation}
From \eqref{u} and \eqref{p} we deduce 
\begin{eqnarray*}
\frac{K_n(\Omega)^p m(t)^{p-p/n}}{\left(-m^{\prime }(t)\right)^{p-1}} &\le&
\mu_1(\Omega) \int_{u_1>t} |u_1|^{p-2}u_1\,dx \qquad \mbox{for a.e.}\> t\ge
t_0 \\
\frac{K_n(\Omega)^p \left(|\Omega|-m(t)\right)^{p-p/n}}{\left(-m^{\prime
}(t)\right)^{p-1}} &\le& \mu_1(\Omega) \int_{u_1>t} |u_1|^{p-2}u_1\,dx
\qquad \mbox{for a.e.}\> t<t_0,
\end{eqnarray*}
that is the claim since $u_1^\star$ is the generalized right-continuous
inverse of $m$.
\end{proof}

Now, let $\tilde{s}=\inf \{s\in (0,|\Omega |):\>u_{1}^{\star }(s)\leq 0\}$.  From \eqref{1star}
we deduce that the function 
\begin{equation}
U(s)=\displaystyle\int_{0}^{s}(u_{1}^{\star }(t))^{p-1}\,dt
\label{U}
\end{equation}
satisfies 
\begin{equation}
\left\{
\begin{array}{ll}
-\left(U'(s)^{1/(p-1)}\right)'\le \dfrac{\mu_1(\Omega)}{K_n(\Omega)^p}\>s^{-p+p/n} \>U(s)^{1/(p-1)}, \qquad s \in (0,\tilde s)
\\ \\
U(0)=U'(\tilde s)=0.
\end{array}
\right.  \label{uu}
\end{equation}
Let $L>0$ be such that $\dfrac{\mu _{1}(\Omega )}{K_{n}(\Omega )^{p}}$
coincides with the first eigenvalue $\sigma _{1}(0,L)$ of the following
Sturm-Liouville problem 
\begin{equation}\label{vv}
\left\{
\begin{array}{ll}
-\left(V'(s)^{1/(p-1)}\right)' = \sigma s^{-p+p/n} \>V(s)^{1/(p-1)}, \qquad s \in (0,L)
\\ \\
V(0)=V'(L)=0.
\end{array}
\right.
\end{equation}
We explicitly observe that such an $L$ always exists. Indeed, let us
consider the following Dirichlet eigenvalue problem 
\begin{equation}
\left\{ 
\begin{array}{ll}
-\Delta _{p}v=\lambda |v|^{p-2}v & \mbox{in}\>B_{R} \\ 
&  \\ 
v=0 & \mbox{on}\>\partial B_{R},
\end{array}
\right.   \label{dirichlet}
\end{equation}
where $B_{R}$ is the ball centered at the origin, having radius $R$. It is
well-known that the first eigenvalue $\lambda _{1}(B_{R})$ of 
\eqref{dirichlet} is simple and that a corresponding eigenfunction does not
change sign in $B_{R}$ (see \cite{L}). Since $\lambda _{1}(B_{R})=\lambda
_{1}(B_{1})R^{-p}$, there exists a unique $R$, say $R=\bar{R}$, such that 
\begin{equation}
\lambda _{1}(B_{\bar{R}})=\left( \dfrac{n\omega _{n}^{1/n}}{K_{n}(\Omega )}
\right) ^{p}\mu _{1}(\Omega ),  \label{l1}
\end{equation}
that is 
\begin{equation}
\bar{R}=\frac{K_{n}(\Omega )}{n\omega _{n}^{1/n}}\left( \frac{\lambda
_{1}(B_{1})}{\mu _{1}(\Omega )}\right) ^{1/p}.  \label{bar_R}
\end{equation}
Now let $v_{1}$ be a positive eigenfunction corresponding to $\lambda
_{1}(B_{\bar{R}})$; it can be proven (see for instance \cite{AFT}) that 
\begin{equation}
V(s)=\displaystyle\int_{0}^{s}(v_{1}^{\star }(t))^{p-1}\,dt  \label{V}
\end{equation}
satisfies \eqref{vv} with $\sigma =\dfrac{\mu _{1}(\Omega )}{K_{n}(\Omega
)^{p}}$ as claimed. Obviously 
\begin{equation*}
L=\omega _{n}{\bar{R}}^{n}=\left( \frac{K_{n}(\Omega )}{n}\right) ^{n}\left( 
\frac{\lambda _{1}(B_{1})}{\mu _{1}(\Omega )}\right) ^{n/p}.
\end{equation*}
We finally observe that problem \eqref{vv} belongs to the class considered
in Lemma \ref{lemma}; it is enough to choose $\gamma =\frac{p}{p-1}$, $\beta
=\frac{p}{p-1}\left( 1-\frac{1}{n}\right) $, $A=L$.

\begin{lemma}\label{L=s}
Let $u_1$ be an eigenfunction corresponding to $\mu_1(\Omega)$ and let $U$, $V$ be defined in \eqref{U}, \eqref{V}.
Then
\begin{equation}
\label{elle}
L\le {\rm min} \left\{{\tilde s}, |\Omega| - {\tilde s}, \dfrac{|\Omega|}{2}\right \};
\end{equation}
moreover

\begin{itemize}

\item[\textit{(i)}] If $L=\tilde s$, then there exists a constant $C\in \R\setminus \{0\}$ such that $U(s)=CV(s)$ and $u_1^\star(s)=Cv_1^\star(s)$, for any $s \in [0,L]$;

\item[\textit{(ii)}] if $L=\frac{|\Omega|}{2}$, then $\tilde s=\frac{|\Omega|}{2}$, $U$ and $V$ are proportional and $u_1^\star=(-u_1)^\star$.
\end{itemize}
\end{lemma}
\begin{proof} 
We firstly prove that $L \leq \tilde s$.
Assume by contradiction that $L > \tilde s$. Then using $U$ as test function in \eqref{uu} we get

\begin{equation}
 \sigma_1(0,\tilde s) \le \frac{\dint_0^{\tilde s}(U'(s))^{\frac{p}{p-1}}ds}{\dint_0^{\tilde s}U(s)^{\frac{p}{p-1}}s^{-p+p/n}ds} \le \frac{\mu_1(\Omega)}{K_n(\Omega)^p}= \sigma_1(0,L),
\end{equation}
which is absurd. 
Now let us show that $L\leq \left\vert \Omega \right\vert -\tilde{s}$.
To this aim note that $w\equiv -u_{1}$ is an eigenfunction corresponding to $\mu_1(\Omega)$
such that $w^{\star }(\left\vert \Omega \right\vert -\tilde{s})=0$. At this point it suffices to repeat the above arguments with $w$ in place
of $u_{1}.$

Summing up the inequalities $L\leq \tilde{s}$ and $L\leq | \Omega |-%
\tilde{s} $ we deduce that $L\leq \left\vert \Omega \right\vert /2.$

$(i)$ Using $U$ and $V$ as test functions in \eqref{uu} and \eqref{vv} respectively, we get
$$
\frac{\dint_0^L (U'(s))^{\frac{p}{p-1}}ds}{\dint_0^LU(s)^{\frac{p}{p-1}}s^{-p+p/n}ds} \le \frac{\mu_1(\Omega)}{K_n(\Omega)^p}=\frac{\dint_0^L (V'(s))^{\frac{p}{p-1}}ds}{\dint_0^LV(s)^{\frac{p}{p-1}}s^{-p+p/n}ds}.
$$
Since $\sigma _{1}(0,L)$,  the first eigenvalue of problem \eqref{vv}, is simple and $L$ has been chosen such that
$\sigma _{1}(0,L)= \frac{\mu _{1}(\Omega )}{K_{n}(\Omega )^{p}}$,  $(i)$ follows.

\noindent $(ii)$ It is an immediate consequence of $(i)$ together with $\dint_0^{|\Omega|}|u_1^\star(t)|^{p-2}u_1^\star(t)\,dt=0.$
\end{proof}

From now on we can assume, without loss of generality, that $L \leq \tilde s \leq |\Omega|/2$.
Another step toward the reverse H\"older inequality is the following
comparison result. 

\begin{proposition}\label{prop2}
Let $u_1$ be an eigenfunction corresponding to $\mu_1(\Omega)$, $q>0$ and $v_{1,q}$ be  a positive eigenfunction of \eqref{dirichlet} corresponding to $\lambda
_{1}(B_{\bar{R}})$  such that   
$$
\displaystyle\int_0^{\tilde s}(u_1^\star(t))^q \,dt = \displaystyle
\int_0^{L} (v_{1,q}^\star(t))^q\,dt.$$
 Then
\begin{equation} 
\displaystyle \int_0^{ s}(u_1^\star(t))^q \,dt \le \displaystyle
\int_0^{s} (v_{1,q}^\star(t))^q\,dt, \qquad s \in [0,L]. \label{UV}\\
\end{equation}

\end{proposition}

\begin{proof}
We can assume that  $L<\tilde s$, since by Lemma \ref{L=s}, part $(i)$, the proposition becomes trivial when $L=\tilde s$.
We first prove \eqref{UV} when $q=p-1$. Denote $U(s)$ as before in \eqref{U} and introduce the function
$$
V_{p-1}(s)=\int_0^{s} (v_{1,p-1}^\star(t))^{p-1}\,dt.
$$
We claim that 
\begin{equation}\label{pippi}
U(s) \le V_{p-1}(s),\qquad s \in [0,L].
\end{equation}
Clearly \eqref{pippi} is fulfilled for $s=0$ or $s=L$. Now, assume by contradiction that there exists $s_{1}\in (0,L)$ such that
\begin{equation*}
U(s_{1})-V_{p-1}(s_{1})=\max_{s \in\left( 0,L\right) }\left( U(s)-V_{p-1}(s)\right) >0.
\end{equation*}
Let
\begin{equation*}
s_{0}=\inf \left\{ t\in (0,L):U(t)-V_{p-1}(t)>0\text{ \ }\forall t\in
(s_{0},s_{1})\right\} .
\end{equation*}
Consider now the functions
\begin{equation*}
\Phi _{1}(s)=\frac{U(s)^{\frac{p}{p-1}}-V_{p-1}(s)^{\frac{p}{p-1}}}{U(s)^{\frac{1}{%
p-1}}}  \label{phiU}
\end{equation*}
and
\begin{equation*}
\Phi _{2}(s)=\frac{U(s)^{\frac{p}{p-1}}-V_{p-1}(s)^{\frac{p}{p-1}}}{V_{p-1}(s)^{\frac{1}{%
p-1}}}.  \label{phiV}
\end{equation*}
Multiplying  (\ref{uu}) by $\Phi _{1}(s)$ and (\ref{vv}) by $\Phi _{2}(s)$
 respectively (note that $\Phi _{1}(s)$ and $\Phi _{2}(s)$
 are positive when $s \in (s_0,s_1)$)  and then subtracting, we get
\begin{equation*}
\int_{s_{0}}^{s_{1}}\left[-u_{1}^{\star }(s)^{\prime }\,\Phi
_{1}(s)+v_{1,p-1}^{\star }(s)^{\prime }\,\Phi
_{2}(s)\right]ds\leq 0.
\end{equation*}
It can be easily checked that
\begin{equation*}
\Phi_1(s_1)-\Phi_2(s_1)=\left(U(s_1)^{\frac{p}{p-1}}-V_{p-1}(s_1)^{\frac{p}{p-1}}\right)\left(U(s_1)^{\frac{1}{1-p}}-V_{p-1}(s_1)^{\frac{1}{1-p}}\right)<0.
\end{equation*}
Hence, since  $\Phi_1(s_0)=\Phi_2(s_0)=0$ and $(u_1^\star(s_1))^{p-1}=U'(s_1)=V_{p-1}'(s_1)=(v_{1,p-1}^\star(s_1))^{p-1}$, an integration by part yields
\begin{equation*} 
\int_{s_{0}}^{s_{1}}\left[u_{1}^{\star }(s)\Phi_{1}'(s)-v_{1,p-1}^{\star }(s)\Phi _{2}'(s)\right]ds \le u_1^\star(s_1)\left(\Phi_1(s_1)-\Phi_2(s_1)\right)<0.
\end{equation*}
We will get a contradiction by showing that 
\begin{equation}\label{i}
I:=\int_{s_{0}}^{s_{1}}\left[u_{1}^{\star }(s)\Phi_{1}'(s)-v_{1,p-1}^{\star }(s)\Phi _{2}'(s)\right]ds\ge 0.
\end{equation}
Indeed, setting 
$$
\tilde U(s)=\left(\frac{u_1^\star(s)}{U(s)}\right)^{\frac{1}{p-1}} \quad \mbox{and}\quad \tilde V_{p-1}(s)=\left(\frac{v_{1,p-1}^\star(s)}{V_{p-1}(s)}\right)^{\frac{1}{p-1}}
$$
a straightforward calculation gives
\begin{eqnarray*}
&I&=\dint_{s_0}^{s_1}\left[(u_1^\star(s))^{p-1}-\frac{p}{p-1}\left(\frac{V_{p-1}(s)}{U(s)}\right)^{\frac{1}{p-1}}u_1^\star(s)(v_{1,p-1}^\star(s))^{p-1}+\frac{1}{p-1}\left(\frac{V_{p-1}(s)}{U(s)}\right)^{\frac{p}{p-1}}(u_1^\star(s))^p+\right.
\\
&+& \left. (v_{1,p-1}^\star(s))^{p-1}-\frac{p}{p-1}\left(\frac{U(s)}{V_{p-1}(s)}\right)^{\frac{1}{p-1}}v_{1,p-1}^\star(s)(u_1^\star(s))^{p-1}+\frac{1}{p-1}\left(\frac{U(s)}{V_{p-1}(s)}\right)^{\frac{p}{p-1}}(v_{1,p-1}^\star(s))^p \right]ds
\\
&=&\dint_{s_0}^{s_1}\left\{U(s)^{\frac{p}{p-1}}\left[{\tilde U(s)}^{\frac{p}{p-1}}-{\tilde V_{p-1}(s)}^{\frac{p}{p-1}}-\frac{p}{p-1}{\tilde V_{p-1}(s)}^{\frac{1}{p-1}}\left(\tilde U(s)- \tilde V_{p-1}(s) \right)\right]+\right.
\\
&+& \left.V_{p-1}(s)^{\frac{p}{p-1}}\left[{\tilde V_{p-1}(s)}^{\frac{p}{p-1}}-{\tilde U(s)}^{\frac{p}{p-1}}-\frac{p}{p-1}{\tilde U(s)}^{\frac{1}{p-1}}\left(\tilde V_{p-1}(s)- \tilde U(s) \right)\right]\right\}ds.
\end{eqnarray*}
The convexity of the function $g(t)=t^{\frac{p}{p-1}}$, $t \ge 0$, ensures that the quantities in the square brackets in the last integral are nonnegative. Therefore the inequality in \eqref{i} is satisfied and finally \eqref{pippi} holds.

Now let $0<q \neq p-1$. Denote
$$U_q(s)= \displaystyle \int_0^{ s}(u_1^\star(t))^q \,dt,\qquad V_q(s)= \displaystyle \int_0^{s} (v_{1,q}^\star(t))^q\,dt.$$
Our aim is to prove that 
$$U_q(s) \le V_q(s) \quad s \in [0,L].$$
As before such an inequality is fulfilled in $s=0$ and in $s=L$.
Assume by contradiction that there exists $s_{1}\in (0,L)$ such that
\begin{equation}
\label{UqVq}
U_q(s_{1})-V_q(s_{1})=\max_{s \in\left( 0,L\right) }\left( U_q(s)-V_q(s)\right) >0.
\end{equation}
Since $s_1 \in (0,L)$, it holds that $U_q'(s_1)=V_q'(s_1)$, that is
\begin{equation}
\label{deriv}
u_1^\star(s_1)= v_{1,q}^\star(s_1).
\end{equation}
Arguing as in the proof of \eqref{pippi}, using \eqref{deriv}, we can prove that 
$$\displaystyle \int_0^{ s}(u_1^\star(t))^{p-1} \,dt \leq  \displaystyle \int_0^{s} (v_{1,q}^\star(t))^{p-1}\,dt,\qquad s \in [0,s_1].$$
This estimate, together with  \eqref{uu}, \eqref{vv} and \eqref{deriv}, gives
$$u_1^\star(s) \leq v_{1,q}^\star(s), \qquad s \in [0,s_1]$$
 and hence  
$U_q(s)\le V_q(s)$ for $s \in [0,s_1],$ which is a contradiction with\eqref{UqVq}.
\end{proof}

Note that the functions $\Phi_{1,2}$ appearing in the proof of the above theorem were also used for example in \cite{L, AFT,BT}.

\begin{theorem}[Reverse H\"older inequality]
\label{reverseH} 
Let $u_1$ be an eigenfunction corresponding to $\mu_1(\Omega)$ and $0<r<q$. There exists a positive constant 
$C=C(n,p,q,r,\mu_1(\Omega),\alpha)$ such that 
\begin{equation}  \label{reverse}
||u_1^+||_{L^q(\Omega)}\le C ||u_1^+||_{L^r(\Omega)}.
\end{equation}
Actually 
\begin{equation*}
C=\frac{||v_1||_{L^q(B_{\bar R})}}{||v_1||_{L^r(B_{\bar R})}},
\end{equation*}
where $v_1$ is any eigenfunction of problem \eqref{dirichlet} in $B_{\bar R}$
corresponding to $\lambda_1(B_{\bar R})$ (see \eqref{l1}, \eqref{bar_R}).
\end{theorem}

\begin{proof}
Let $v_{1,r}$ be the eigenfunction of problem \eqref{dirichlet} corresponding to $\lambda_1(B_{\bar R})$ satisfying 
$$||v_{1,r}||_{L^r(B_{\bar R})}=||u_1^+||_{L^r(\Omega)}.$$ We define $v_{1,r}^\star(s)=0$ for $s \in [L,\tilde s]$. Proposition \ref{prop2} immediately implies
$$
\int_0^s (u_1^\star(t))^r\,dt \le \int_0^s(v_{1,r}^\star(t))^r\, dt, \qquad s\in [0,\tilde s], \qquad\mbox{and}\qquad \int_0^{\tilde s}(u_1^\star(t))^r\,dt=\int_0^{\tilde s}(v_{1,r}^\star(t))^r\,dt.
$$
By well-known properties of rearrangements (see for instance \cite{ALT}) we get
$$
||u_1^+||_{L^q(\Omega)}=\int_0^{\tilde s}(u_1^\star(t))^q\,dt\le \int_0^{\tilde s} (v_{1,r}^\star(t))^q\,dt=||v_{1,r}||_{L^q(B_{\bar R})}.
$$
Finally 
\begin{equation*}
||u_1^+||_{L^q(\Omega)}\le || v_{1,r}||_{L^q(B_{\bar
R})}=||u_1^+||_{L^r(\Omega)} \frac{|| v_{1,r}||_{L^q(B_{\bar R})}}{|| v_{1,r}||_{L^r(B_{\bar R})}}= ||u_1^+||_{L^r(\Omega)} \frac{||
v_1||_{L^q(B_{\bar R})}}{||v_1||_{L^r(B_{\bar R})}}.
\end{equation*}
\end{proof}

\section{Proof of Theorem \ref{main} and comparison with previous results}

\subsection{Proof of Theorem \ref{main}}
Set $\Omega^+=\{x \in \Omega: \> u_1(x)>0\}$ and suppose as in Section 2 that $|\Omega^+| \le \frac{|\Omega|}{2}.$ From \eqref{reverse}
using H\"older inequality we have 
\begin{equation}  \label{omega}
|\Omega^+|^{1/q-1/r}\le \frac{||v_1||_{L^q(B_{\bar R})}}{||v_1||_{L^r(B_{
\bar R})}},
\end{equation}
where $v_1$ is as in Theorem \ref{reverseH}. We choose 
\begin{equation*}
v_1(x)=\Psi_p\left(\left(\frac{\mu_1(\Omega)}{\alpha}\right)^{1/p}|x|\right),
\end{equation*}
where 
\begin{equation*}
\alpha=\left(\frac{K_n(\Omega)}{n\omega_n^{1/n}}\right)^p 
\end{equation*}
and $\Psi_p(r)$ is the solution to the following Sturm-Liouville problem 
\begin{equation*}
\left\{\begin{array}{ll}
-(p-1)|\Psi_p'|^{p-2}\Psi_p''-\frac{n-1}{r}|\Psi_p'|^{p-1}=\Psi_p^{p-1} & \mbox{in}\> (0,\psi_{p})
\\ \\
\Psi_p'(0)=\Psi_p(\psi_p)=0,
\end{array}
\right.
\end{equation*}
normalized in such a way that $\Psi_p(0)=1$, where $\psi_p$ is the first
positive zero of $\Psi_p$. Clearly, in the linear case $p=2$, $\Psi_p(r)$
coincides with $r^{1-n/2}J_{n/2-1}(r)$ and $\psi_p$ is the first positive
zero $j_{n/2-1,1}$ of the Bessel function of the first kind $J_{n/2-1}$.
With this choice of $w_1$ we get that $\bar R=\psi_p\left(\frac{\alpha}{
\mu_1(\Omega)}\right)^{1/p}$ and \eqref{omega} becomes 
\begin{eqnarray}  \label{w}
\frac{||w_1||_{L^q(B_{\bar R})}}{||w_1||_{L^r(B_{\bar R})}} &=&\frac{
\left(n\omega_n\displaystyle\int_0^{\bar R} t^{n-1}\Psi_p\left(\left(\frac{
\mu_1(\Omega)}{\alpha}\right)^{1/p}t\right)^qdt\right)^{1/q}}{\left(n\omega_n
\displaystyle\int_0^{\bar R} t^{n-1}\Psi_p\left(\left(\frac{\mu_1(\Omega)}{
\alpha}\right)^{1/p}t\right)^rdt\right)^{1/r}} \\
&=&(n\omega_n)^{1/q-1/r}\left(\frac{\alpha}{\mu_1(\Omega)}
\right)^{n/(pq)-n/(pr)}\frac{\left(\displaystyle\int_0^{\psi_p}t^{n-1}
\Psi_p(t)^qdt\right)^{1/q}}{\left(\displaystyle\int_0^{\psi_p}t^{n-1}
\Psi_p(t)^rdt\right)^{1/r}},  \notag
\end{eqnarray}
that is 
\begin{equation}  \label{b}
\mu_1(\Omega) \ge \alpha\left(\frac{n\omega_n}{|\Omega^+|} \right)^{p/n}
\frac{\left(\displaystyle\int_0^{\psi_p}t^{n-1}\Psi_p(t)^rdt
\right)^{pq/n(q-r)}}{\left(\displaystyle\int_0^{\psi_p}t^{n-1}\Psi_p(t)^qdt
\right)^{pr/n(q-r)}}.
\end{equation}
Let 
\begin{equation}  \label{f}
f(s)=\left(\frac{\int_0^{\psi_p}t^{n-1}\Psi_p(t)^sdt}{\int_0^{
\psi_p}t^{n-1}dt}\right)^{1/s}=\left(\frac{n}{\psi_p^n}\int_0^{
\psi_p}t^{n-1}\Psi_p(t)^sdt\right)^{1/s};
\end{equation}
it is easy to prove that 
\begin{equation}  \label{sup}
\sup_{0<r<q}\left(\frac{f(r)}{f(q)}\right)^{pqr/n(q-r)}=1;
\end{equation}
recalling that $|\Omega^+| \le |\Omega|/2$, \eqref{b} implies 
\begin{equation*}
\mu_1(\Omega) \ge 2^{p/n} \alpha\> \frac{\psi_p^p}{\left(\frac{|\Omega|}{
\omega_n}\right)^{p/n}} 
\end{equation*}
and estimate \eqref{estimate} immediately follows.

It remains to prove that, when $n=p=2$, \eqref{estimate} is sharp. To this
aim let us consider the sequence of rhombi $\Omega _{m}$ having vertices $A_m
$, $B_m$, $C_m$, $D_m$ (see Figure 1), acute angles $\beta _{m} = \frac{2\pi 
}{m}$ ($m>4$) and sides with length one.

 In \cite{Ci} it is proved that 
\begin{equation*}
K_2(\Omega_m)=\sqrt{2\sin \beta_m} 
\end{equation*}
and hence 
\begin{equation*}
\alpha _{m}\equiv \left(\frac{K_2(\Omega_m)}{2\sqrt{\pi}}\right)^2= \frac{
\sin \beta _{m}}{2\pi }.  \label{alf_m}
\end{equation*}
Let $u_{m}$ be an eigenfunction corresponding to $\mu_1(\Omega_m)$. By
Proposition 4.1 in \cite{BCT} the nodal line of $u_{m}$ is the shortest
diagonal of the rhombus, the segment $\overline{B_mD_m}$, and $u_{m}$ is odd
with respect to $\overline{B_mD_m}$; let us denote by $T_{m}$ the triangle of vertices $A_m$, $B_m$, $D_m$. Clearly the restriction of $u_m$ to $T_{m}$ is
an eigenfunction corresponding to the first eigenvalue $\lambda_{1}^{DN}(T_m)$ of
the following  problem with mixed boundary conditions
\begin{equation*}  \label{triangolo}
\left\{
\begin{array}{ll}
-\Delta u_m =\lambda^{DN}u_m & \mbox{in}\,\, T_m \\ 
&  \\ 
u_m=0 & \mbox{on} \,\,\overline{B_mD_m} \\ 
&  \\ 
\dfrac{\partial u_m}{\partial \nu}=0 & \mbox{on}\,\, \partial T_m \setminus 
\overline{B_mD_m},
\end{array}
\right.
\end{equation*}
and $\lambda_{1}^{DN}(T_m)= \mu_1(\Omega_m)$. 
Let $S_{m}^{e}$ be the sector, centered at $A_m$, having radius one and opening angle $\beta_m$,
containing $T_m$, and let $S_{m}^{i}$ be the sector, centered at $A_m$, with opening angle $
\beta_m$, tangent to $T_m$ at the midpoint of the segment $\overline{B_mD_m}$ 
(see Figure 1).  

\noindent In what follows we mean with $\lambda_{1}^{DN}(S_{m}^{e})$  
($\lambda_{1}^{DN}(S_{m}^{i})$) 
the first eigenvalue of the Laplace operator in $S_{m}^{e}$ 
($S_{m}^{i}$)
with homogeneous Dirichlet boundary conditions on the segments $\overline{A_mB_m}$,  $\overline{A_mD_m}$ and homogeneous Neumann boundary conditions on the remaining part of 
$\partial S_{m}^{e}$  ($\partial S_{m}^{i}$). 
Any eigenfunction corresponding to $\lambda_{1}^{DN}(T_m)$ can be used as test function for $\lambda_{1}^{DN}(S_{m}^{e})$ after setting its value equal to zero on $S_{m}^{e} \setminus T_m$. In the same way we can use any eigenfunction corresponding to $\lambda_{1}^{DN}(S_{m}^{i})$ as test function in $\lambda_{1}^{DN}(T_m)$, and therefore
\begin{equation*}  \label{stima}
 j_{0,1}^2=
\lambda_{1}^{DN}(S_{m}^{e})
 \le \lambda_{1}^{DN}(T_m) \le 
\lambda_{1}^{DN}(S_{m}^{i})=
\frac{j_{0,1}^2}{\cos^2\left(\frac{\beta_m}{2}\right)},
\end{equation*}
 being $j_{0,1}$ the first positive zero of the Bessel function of the first
kind $J_0$. 
Hence 
\begin{equation*}
\lim_{m\to +\infty} \mu_1(\Omega_m)=\lim_{m\to +\infty} \lambda_{1}^{DN}(T_m) =
j_{0,1}^2.
\end{equation*}
On the other hand $\lambda_1(\Omega_m^\sharp)=\dfrac{\pi j_{0,1}^2}{\sin
\beta_m}$ and then 
\begin{equation*}
\lim_{m \to +\infty} \frac{\mu_1(\Omega_m)}{\alpha_m\lambda_1(\Omega_m^\sharp)}=2. 
\end{equation*}

\subsection{Comparison of estimate \eqref{estimate} with previous results available in literature}

 We begin by showing that  \eqref{estimate} improves the following bound for $p=2$
(see Corollay 3.1. in \cite{BCT}): 
\begin{eqnarray*}
\mu _{1}(\Omega)&\geq \dfrac{\alpha}{|\Omega|^{2/n}}  &\left(2n\omega _{n}\int_{0}^{j_{\frac{n}{2}-1,1}}s^{\frac{n}{2}}J_{\frac{n}{2}%
-1}\left( s\right) ds\right)^{2/n}\times
\\
&&\times\exp \left( \frac{\frac{2}{n}\int_{0}^{j_{\frac{n}{2}-1,1}}J_{%
\frac{n}{2}-1}\left( s\right) s^{\frac{n}{2}}\left[ \left( \frac{n}{2}%
-1\right) s-\log \left( J_{\frac{n}{2}-1}\left( s\right) \right) \right] ds}{%
\int_{0}^{j_{\frac{n}{2}-1,1}}s^{\frac{n}{2}}J_{\frac{n}{2}-1}\left(
s\right) ds}\right). 
\end{eqnarray*}
Such an inequality can be rewritten in terms of the function $f$ defined in \eqref{f}: 
\begin{equation*}
\mu_1(\Omega)\ge 2^{2/n}\>\alpha\>\left[ \sup_{1 \le q}\left(\frac{f(1)}{f(q)
}\right)^{2q/n(q-1)}\right]\frac{j_{n/2-1,1}^2}{\left(\frac{|\Omega|}{
\omega_n}\right)^{2/n}},
\end{equation*}
while \eqref{estimate} reads as 
\begin{equation*}
\mu_1(\Omega)\ge 2^{2/n}\>\alpha \>\left[\sup_{0<r<q}\left(\frac{f(r)}{f(q)}
\right)^{2qr/n(q-r)}\right]\frac{j_{n/2-1,1}^2}{\left(\frac{|\Omega|}{
\omega_n}\right)^{2/n}}.
\end{equation*}
Taking into account \eqref{sup} we get the claim.

\medskip

When $p \ge 2$, estimate \eqref{estimate} is better
than the one contained in \cite{A,AM}. In these last papers the authors
prove that 
\begin{equation*}
\mu_1(\Omega) \ge 2^{p/n}\left(\frac{n}{p(n-1)}\right)^p\frac{K_n(\Omega)^p}{
|\Omega|^{p/n}}. 
\end{equation*}
Inequality \eqref{estimate} can be also read as 
\begin{equation*}
\mu_1(\Omega) \ge 2^{p/n} \left(\frac{\psi_p}{n}\right)^p \frac{K_n(\Omega)^p
}{|\Omega|^{p/n}}. 
\end{equation*}
Hence, in order to get our claim, it is enough to verify that 
\begin{equation}\label{psi}
\psi_p >\frac{n^2}{p(n-1)}, \qquad p\ge 2,\quad n \ge 2. 
\end{equation}
In \cite{L1} (see also \cite{AB}) it is proved that 
\begin{equation*}  \label{lind}
q\,\psi_q\le p \,\psi_p, \qquad q \le p;
\end{equation*}
then, choosing $q=2$ in \eqref{lind} we get 
\begin{equation}  \label{ll}
\psi_p > \frac{2}{p} j_{n/2-1,1}.
\end{equation}
On the other hand in \cite{Lo} the author proves that 
\begin{equation}  \label{lorch}
j_{n/2-1,1}^2 > \frac{n}{2}\left(\frac{n}{2}+4\right).
\end{equation}
Gathering \eqref{ll} and \eqref{lorch} we immediately get \eqref{psi}.

\medskip

Finally consider the class $\mathcal{G}$ of planar, convex domains $\Omega$ that are symmetric about a point. A result contained in \cite{Ci} ensures that 
$$
K_2(\Omega)^2=\frac{2w(\Omega)^2}{|\Omega|}, \qquad \forall \Omega \in \mathcal{G},
$$
where $w(\Omega)$ stands for the width of $\Omega$. In such a class of domains our estimate \eqref{estimate} reads as

\begin{equation}
\label{e2}
\mu_1(\Omega) \ge   j_{0,1}^2 \frac{w(\Omega)^2}{|\Omega|^2},  \qquad \forall \Omega \in \mathcal{G}.
\end{equation} 
Then for any $\Omega \in \mathcal{G}$ such that
\begin{equation}\label{e3}
|\Omega|<C w(\Omega)d(\Omega), \quad \mbox{with}\>\> 0<C<\frac{j_{0,1}}{\pi},
\end{equation}
 estimate \eqref{e2} improves the classical Payne-Weinberger inequality \eqref{PW}. Indeed, \eqref{e2} and \eqref{e3} immediately imply
 $$
 \mu_1(\Omega)d(\Omega)^2 \ge  \frac{j_{0,1}^2}{C^2} > \pi^2.
 $$

\bigskip
\begin{figure}[h]
\centering
\includegraphics[scale=0.6]{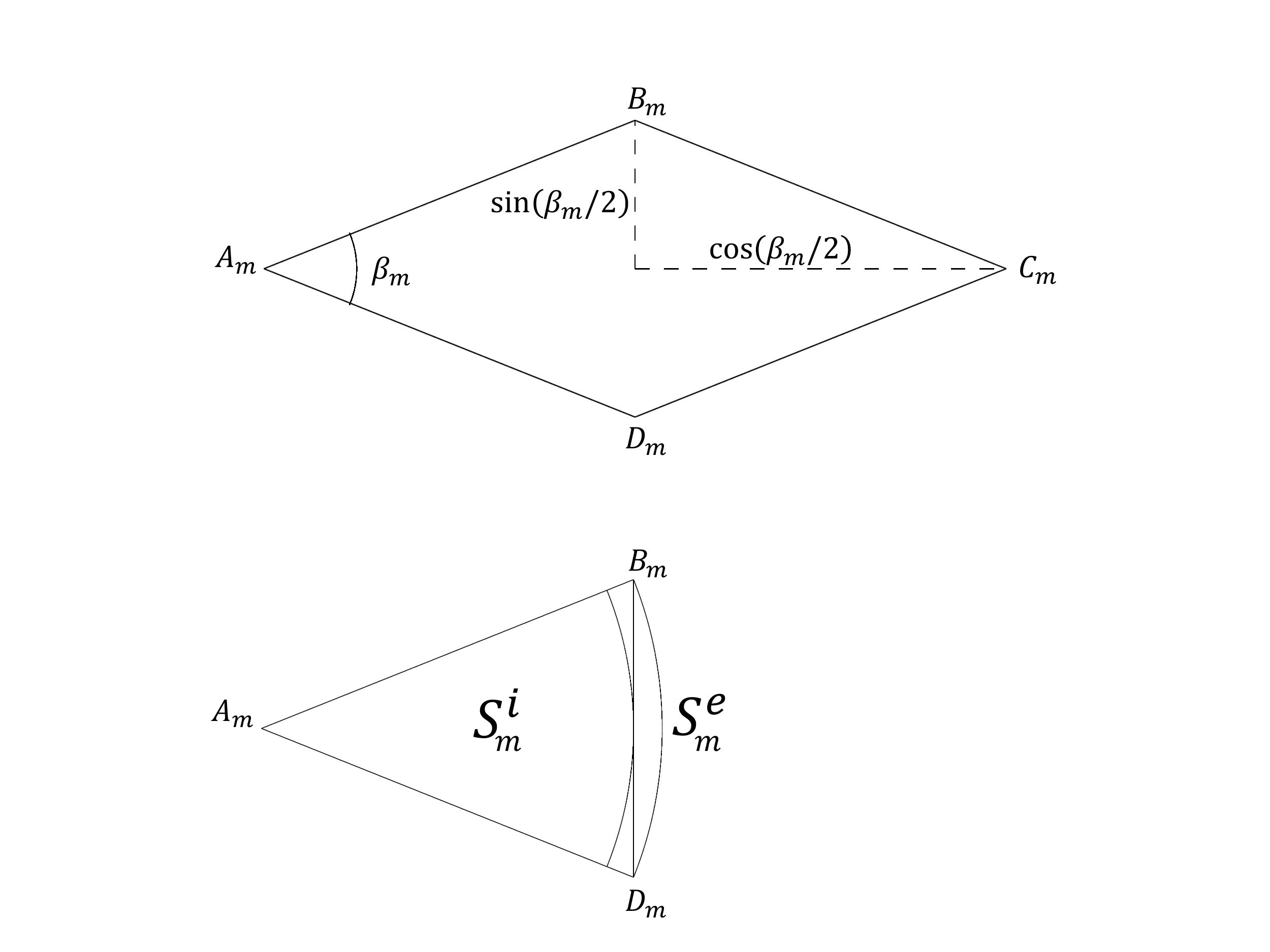} \begin{caption}
{}
\end{caption}
\end{figure}

\end{document}